\documentclass[12pt]{amsart}
\usepackage{amsfonts, amssymb, amsmath, amsthm, euscript}
\newtheorem*{thm}{Theorem}

\def\Closed{{\mathcal Closed}}
\def\C{{\mathcal C}}
\def\D{{\mathcal D}}
\def\Hom{\text{Hom}}
\def\id{\text{id}}

\def\F{{\mathbb F}}
\def\Vec{{\mathcal Vec}}
\def\Grp{{\mathcal Grp}}
\def\Set{{\mathcal Set}}
\def\Top{{\mathcal Top}}

\begin{document}

\title[]{Continuity is an Adjoint Functor}

\author{Edward S. Letzter}

\address{Department of Mathematics\\
        Temple University\\
        Philadelphia, PA 19122-6094}
      
      \email{letzter@temple.edu }

    \date{January 1, 2014}

\maketitle


\noindent {\bf 1. INTRODUCTION.} The emergence of category theory,
introduced by S. Eilenberg and S. Mac Lane in the 1940s
(cf.~\cite{EilenbergMacLane}), was among the most important
mathematical developments of the twentieth century. The profound
impact of the theory continues to this day, and categorical methods
are currently used, for example, in algebra, geometry, topology,
mathematical physics, logic, and theoretical computer science. (Hints
of this breadth can be found, e.g., in \cite{CoeckePaquette},
\cite{MacLane}, and \cite{Pierce}).

A \emph{category\/} is comprised of \emph{objects} and
\emph{morphisms} betweeen objects.  Standard elementary examples
include $\Set$, where the objects are sets and the morphisms are set
functions; $\Grp$, where the objects are groups and the morphisms are
group homomorphisms; and $\Top$, where the objects are topological
spaces and the morphisms are continuous functions.

Relationships among different categories are established via
\emph{functors} between them. A pair of \emph{adjoint functors}, as
formulated in 1958 by D. M. Kan \cite{Kan}, determines a particularly
close tie between two categories. Adjoint functors are essential tools
in category theory, and their introduction was a significant milestone
in its development.

While it is commonly held that adjoint functors ``occur almost
everywhere'' \cite[p.~107]{MacLane}), at least in many areas of
mathematics, the typical first examples presented to students may not
immediately reveal the fundamental importance of the ideas
involved. (One such typical example, a left adjoint to a ``forgetful
functor,'' is described at the end of the brief review provided in the
next section.)
 
Our aim in this note, then, is to illustrate how a natural example of
adjoint functors can be ``found'' in the definition of a continuous
map between topological spaces.  In particular, we show, for a set
function $\varphi:X\rightarrow Y$ between topological spaces $X$ and
$Y$, that $\varphi$ is continuous if and only if certain naturally
arising functors are adjoint.

\medskip\noindent {\bf Remarks.}  The main result presented in this
note was recorded in a more abstract, and apparently more obscure,
setting in \cite{Letzter}. Moreover, in noncommutative algebraic
geometry, certain adjoint functor pairs serve as morphisms between
(not explicitly defined) noncommutative spaces. (This approach follows
\cite{Rosenberg} and \cite{VandenBergh}; see also \cite{Smith}.)

The reader is referred for example to \cite{LawvereSchanuel} for an
accessible general introduction to category theory and its history.

\bigskip\noindent {\bf 2. A BRIEF RREVIEW.} As mentioned above, a
category $\C$ consists of \emph{objects\/} and \emph{morphisms\/}
between objects. The morphisms in $\C$ from an object $A$ to an object
$B$ comprise a set denoted $\Hom_\C(A,B)$. A morphism in $\C$ from $A$
to $B$ is also referred to as a \emph{$\C$-morphism\/} and denoted $A
\rightarrow B$.

These morphisms must satisfy the following conditions:

\medskip\noindent (1) For each pair of $\C$-morphisms $j:D\rightarrow
E$ and $k:E\rightarrow F$, there is a \emph{composition\/} morphism
$k\circ j:D\rightarrow F$, such that
\[\ell \circ (k \circ j) \quad \text{and} \text \quad (\ell \circ
k)\circ j\]
produce the same morphism $D \rightarrow G$, for all $\C$-morphisms
$j:D\rightarrow E$, $k:E\rightarrow F$, and $\ell:F\rightarrow G$.

\medskip\noindent (2) For each object of $\C$ there is an identity
morphism $\id_\C$ from that object to itself, such that the compositions
\[ A  \stackrel{\id_\C}\longrightarrow A \stackrel{f}\longrightarrow B
\quad \text{and} \quad A \stackrel{f}\longrightarrow B
\stackrel{\id_\C}\longrightarrow B \]
are both equal to $f$, for all $\C$-morphisms $f:A\rightarrow B$.

\medskip\noindent {\bf Adjoint Functors.} Let $\C$ and $\D$ be categories. A
\emph{(covariant) functor} $\Phi:\C \rightarrow \D$ assigns to each object $A$ of
$\C$ an object $\Phi(A)$ of $\D$, and to each $\C$-morphism $f:A
\rightarrow B$ a $\D$-morphism
\[\Phi(f):\Phi(A)\rightarrow \Phi(B),\]
such that
  \[ \Phi(\id_\C) = \id_\D \quad \text{and} \quad \Phi(\ell \circ k) =
  \Phi(\ell)\circ \Phi(k),\]
  for all $\C$-morphisms $k:E\rightarrow F$ and $\ell:F\rightarrow G$.

Now consider a pair of functors 
\[\Phi:\C\rightarrow \D \quad \text{and} \quad \Psi:\D\rightarrow
\C.\] 
Also suppose, for all objects $L$ of $\C$ and $M$ of $\D$, that there
exists a bijective function
  \[\Hom_\D(\Phi(L), M) \stackrel{\beta}\longrightarrow
  \Hom_\C(L, \Psi(M)),\]
  assigning to each $\D$-morphism 
\[r:\Phi(L) \rightarrow M\]
a $\C$-morphism 
\[\beta(r):L\rightarrow \Psi(M).\]
We then say that $(\Phi,\Psi)$ is an \emph{adjoint pair\/} provided
\[\beta (t \circ r) = \Psi(t) \circ \beta(r) \quad \text{and} \quad \beta (r \circ
\Phi(s)) = \beta (r) \circ s , \leqno{(\ast)}\]
for all $\C$-morphisms $s:L' \rightarrow L$ and all $\D$-morphisms
$t:M\rightarrow M'$.  We indicate that the bijection $\beta$ satisfies
the two conditions in $(\ast)$ by saying that $\beta$ is
\emph{natural in $L$ and $M$}.

\medskip\noindent{\bf Example.} The following is a standard first
example of an adjoint pair: Let $\F$ be a field, and let $\Vec_\F$
denote the category whose objects are $\F$-vector spaces and whose
morphisms are $\F$-linear transformations.  Let $\Psi:\Vec_\F
\rightarrow \Set$ be the ``forgetful functor,'' assigning to each
vector space its underlying set of vectors, and assigning to each
linear transformation its underlying set function. Let $\Phi:\Set
\rightarrow \Vec_\F$ be the functor assigning to each set $S$ the
$\F$-vector space $\F S$ with basis $S$, and assigning to each set map $S
\rightarrow S'$ the $\F$-linear extension $\F S \rightarrow \F S'$.
Then $(\Phi,\Psi)$ is an adjoint pair.  (Details and analogous
examples can be found, e.g., in Chapter IV of \cite{MacLane}.)

\bigskip\noindent {\bf 3. CONTINUITY VIA ADJOINT PAIRS.} For the
remainder, let $X$ and $Y$ be topological spaces, and let
$\varphi:X\rightarrow Y$ be a set function. Recall that $\varphi$ will
be continuous if and only if $\varphi^{-1}(V)$ is closed in $X$ for
all closed subsets $V$ of $Y$.

\medskip\noindent{\bf The category of closed subsets of a topological
  space.} Define $\Closed (X)$ to be the category whose objects are
closed subsets of $X$ and whose morphisms are described as follows:
Let $U$ and $U'$ be closed subsets of $X$. If $U$ is a subset of $U'$
then there is exactly one morphism, the inclusion map, from $U$ to
$U'$. If $U$ is not a subset of $U'$ then the set of morphisms from
$U$ to $U'$ is empty. Similarly define $\Closed (Y)$.

\medskip\noindent{\bf An adjointness criterion.} Adjoint pairs of
functors between the categories $\Closed(X)$ and $\Closed(Y)$ can be
described in a particularly simple way, as follows: Suppose that
\[\Phi:\Closed(X) \rightarrow \Closed(Y) \quad \text{and} \quad
\Psi:\Closed(Y) \rightarrow \Closed(X)\]
are functors. Then for $U$ in $\Closed (X)$ and $V$ in $\Closed(Y)$, there
exists a bijection
  \[ \Hom_{\Closed(Y)}(\Phi(U),V) \quad
  \stackrel{\beta}\longrightarrow \quad
  \Hom_{\Closed(X)}(U,\Psi(V)) \]
  exactly when one (and only one) of the following two cases holds: 
\medskip

\begin{description}
\item[{\rm Case 1}] Both $\Phi(U) 
  \subseteq V$ and $V \subseteq \Psi(V)$
\medskip

\item[{\rm Case 2}] Both $\Phi(U) \nsubseteq V$ and $U \nsubseteq \Psi(V)$
\end{description}
\medskip

\noindent It is not hard to verify that if a bijection $\beta$ as
above does exist, then it must be natural in $U$ and $V$, in the sense
of $(\ast)$. (Also note that if $\beta$ exists it must be unique.) We
deduce that $(\Phi,\Psi)$ is an adjoint pair exactly when the
statement
\[\Phi(U) \subseteq V \quad \text{if and only if} \quad U \subseteq
\Psi(V) \leqno{(\dag)} \]
holds true for all $U$ in $\Closed (X)$ and $V$ in $\Closed(Y)$. 

\medskip\noindent{\bf From a function to a pair of functors.}
Consider the assignments
\[T_\varphi : \Closed(X) \; \longrightarrow \; \Closed(Y), \quad U
\longmapsto \overline{\varphi(U)},\]
and
\[T^\varphi : \Closed(Y) \; \longrightarrow \; \Closed(X), \quad V
\longmapsto \overline{\varphi^{-1}(V)},\]
where $\overline{S}$ denotes the closure of an arbitrary subset of $X$
or $Y$. It is straightforward to check that $T_\varphi$ and $T^\varphi$ are
functors. 

Our aim now is to prove:

\begin{thm} The function $\varphi$ is continuous if and only if
  $(T_\varphi, T^\varphi)$ is an adjoint pair. \end{thm}

\begin{proof} To start, we claim that the following four conditions
  are equivalent, for all $U$ in $\Closed(X)$ and $V$ in $\Closed (Y)$:
\medskip

\begin{enumerate}

\item \label{equivalenceone} \quad $(T_\varphi, T^\varphi)$ \; is an adjoint pair.
\medskip

\item \label{equivalencetwo} \quad $ T_\varphi(U) \subseteq V$ \; if and only if \; $U \subseteq
T^\varphi(V)$.
\medskip

\item \label{equivalencethree} \quad $\overline{\varphi(U)} \subseteq
  V$ \; if and only if \; $U \subseteq \overline{\varphi^{-1}(V)}$.
  \medskip

\item \label{equivalencefour} \quad $\varphi(U) \subseteq V$ \; if and
  only if \; $U \subseteq \overline{\varphi^{-1}(V)}$.
\end{enumerate}
\medskip

The equivalence of (\ref{equivalenceone}), (\ref{equivalencetwo}), and
(\ref{equivalencethree}) follows directly from
$(\dag)$. To see why (\ref{equivalencethree}) is
equivalent to (\ref{equivalencefour}), recall that the closure of
$\varphi(U)$ in $Y$ is the smallest closed subset 
containing $\varphi(U)$, and so
\[\varphi(U) \subseteq V \quad \text{if and only if} \quad  \overline{\varphi(U)}
\subseteq V.\]

Next, it is also true, for all $U$ in $\Closed(X)$ and $V$ in
$\Closed(Y) $, that if $\varphi(U) \subseteq V$ then
\[U \; \subseteq \; \varphi^{-1}(\varphi(U)) \; \subseteq \; \varphi^{-1}(V)
\; \subseteq \; \overline{\varphi^{-1}(V)}.\]
Hence, it follows from the equivalence of (\ref{equivalenceone}) and
(\ref{equivalencefour}) above that $(T_\varphi, T^\varphi)$ is an
adjoint pair exactly when
\[ U \; \subseteq \; \overline{\varphi^{-1}(V)} \; \Longrightarrow \;
\varphi(U) \; \subseteq V , \leqno{(\ddag)}\]
for all $U$ in $\Closed(X)$ and $V$ in $\Closed(Y)$.

Now suppose that $\varphi$ is a continuous function. As noted above,
$\varphi^{-1}(V)$ is closed in $X$ for all closed subsets $V$ of $Y$,
and so
\[ \varphi^{-1}(V) = \overline{\varphi^{-1}(V)}.\]
Therefore, for all $U$ in $\Closed(X)$ and $V$ in $\Closed(Y)$, if
\[ U \; \subseteq \; \varphi^{-1}(V) \; = \; \overline{\varphi^{-1}(V)},\]
then
\[ \varphi(U) \; \subseteq \; \varphi(\varphi^{-1}(V)) \; \subseteq \; V.\]
Consequently, $(T_\varphi,T^\varphi)$ is an adjoint pair.

Conversely, suppose that $(T_\varphi,T^\varphi)$ is an adjoint pair,
and fix an arbitrary $V$ in $\Closed(Y)$. Set
\[ U \; := \; \overline{\varphi^{-1}(V)}.\]
By $(\ddag)$, 
\[ \varphi(U) \; \subseteq V, \; \]
and so
\[ \overline{\varphi^{-1}(V)} \; = \; U \subseteq
\; \varphi^{-1}(\varphi(U)) \; \subseteq \; \varphi^{-1}(V).\]
Therefore,
\[ \varphi^{-1}(V) \; = \; \overline{\varphi^{-1}(V)}.\]
In particular, $\varphi^{-1}(V)$ is closed, and so $\varphi$ is
continuous. The theorem follows. 
\end{proof}

\medskip\noindent{\bf Acknowledgment.} I am grateful to the referees
for their helpful comments.


\end{document}